\documentclass[a4paper,11pt]{article}
\usepackage{amsmath}
\usepackage{amsfonts}
\usepackage{amssymb}
\usepackage{amsthm}
\usepackage{breqn}
\usepackage{setspace}
\usepackage{fullpage}
\usepackage{enumitem}
\usepackage{comment}
\usepackage{hyperref}
\usepackage{bbm}
\usepackage{tikz}
\usepackage{enumitem}
\usepackage{mathtools} 
\usepackage[utf8]{inputenc}
\usepackage[english]{babel}
\usepackage{mathtools}
\usetikzlibrary{patterns,arrows,decorations.pathreplacing}
\newtheorem{lemma}{Lemma}
\newtheorem{theorem}{Theorem} 
\newtheorem{corollary}{Corollary} 
\newtheorem{definition}{Definition}
\newtheorem{claim}{Claim}
\newtheorem{conjecture}{Conjecture}

\newtheorem{proposition}{Proposition}

\usepackage{authblk}
\def \s{\subset}

\date{}
\begin{document}
\title{The maximum number of induced $C_5$'s in a planar graph}
\author[1]{Debarun Ghosh}
\author[1,2]{Ervin Gy\H{o}ri} 
\author[3]{Oliver Janzer}
\author[1]{Addisu Paulos}
\author[1,2]{Nika Salia}
\author[1,4]{Oscar Zamora} 
\affil[1]{Central European University, Budapest\par
\texttt{oscarz93@yahoo.es, addisu_2004@yahoo.com, ghosh_debarun@phd.ceu.edu}}
\affil[2]{Alfr\'ed R\'enyi Institute of Mathematics, Budapest \par
\texttt{gyori.ervin@renyi.mta.hu, nika@renyi.hu }}
\affil[3]{Department of Mathematics, ETH Z\"urich, Switzerland\par \texttt{oliver.janzer@math.ethz.ch}}
\affil[4]{Universidad de Costa Rica, San Jos\'e}
\maketitle
\begin{abstract}
Finding the maximum number of induced cycles of length $k$ in a graph on $n$ vertices has been one of the most intriguing open problems of Extremal Graph Theory.  Recently Balogh, Hu, Lidick\'{y} and Pfender answered the question in the case $k=5$.  In this paper we determine precisely, for all sufficiently large $n$, the maximum number of induced $5$-cycles that an $n$-vertex planar graph can contain.
\end{abstract}

\section{Introduction}
Tur\'{a}n-type problems are central in extremal combinatorics, with many long-standing open questions. A typical problem in the area is to determine the maximum number of graph-theoretical structures like edges, triangles or cycles in host graphs without certain substructures.

The problem of maximizing the number of induced copies of a fixed small graph $H$ has attracted a lot of attention recently, see, for example, \cite{4vertex,4vertex2,triangles}.
Morrison and Scott determined the maximum possible number of induced cycles, without restriction on length, that can be contained in a graph on $n$ vertices \cite{indcycle}. The maximal number of induced complete bipartite graphs and induced complete $r$-partite subgraphs have also been studied \cite{bol1,bol2,com_bi_par}.   The problem of determining the maximum number of induced $C_5$'s has been elusive for a long time and was finally solved by  Balogh, Hu, Lidick\'{y} and Pfender \cite{indc5}.

A planar graph is a graph that can be  drawn on the plane in such a way that its edges intersect only at their endpoints. Such a drawing of a planar graph is called a plane graph.

Let the maximum number of (not necessarily induced) copies of the fixed graph $H$ in an $n$-vertex planar graph be denoted by $f(n,H)$. The study of this function was initiated by Hakimi and Schmeichel \cite{hakimi}, who investigated the case where $H$ is a cycle. They determined the value of $f(n,C_3)$ and $f(n,C_4)$ precisely and showed that in general $f(n,C_k)=\Theta(n^{\lfloor k/2\rfloor})$. Their formula for $f(n,C_4)$ is as follows.  
\begin{theorem} (Hakimi  and Schmeichel \cite{hakimi}) \label{thm:hakimi}
$f(n,C_4)=\frac{1}{2}(n^2+3n-22)$ for $n\geq 4$.
\end{theorem}

Recently, Gy\H ori et al. \cite{c5} determined the exact answer for the $5$-cycle.
\begin{theorem} (Gy{\H o}ri et al. \cite{c5})
    For $n = 6$ and $n \geq 8$, $f(n, C_5) = 2n^2 - 10n + 12$.  
For $n = 5$ we have $f(n, C_5) = 6$, and for $n = 7$ we have $f(n, C_5) = 41$.
\end{theorem}

Alon and Caro \cite{AC84} determined $f(n,H)$ precisely for various graphs $H$, including the complete bipartite graph $K_{2,t}$ for any $t$.

Very recently, Huynh, Joret and Wood \cite{HJW20} determined the order of magnitude of $f(n,H)$ for every graph $H$.

In this paper, we initiate the study of the variant of the above problem where we are trying to maximize the number of \emph{induced} copies of a certain fixed graph $H$ in an $n$-vertex planar graph. Since every triangle is induced, the answer is determined by the result of Hakimi and Schmeichel for $H=C_3$. For $H=C_4$, note that the planar graph $K_{2,{n-2}}$ contains exactly $\frac{1}{2}(n^2-5n+6)$ induced 4-cycles. It follows from this observation and Theorem \ref{thm:hakimi} that the maximum number of induced 4-cycles in a planar graph with $n$ vertices is $\frac{1}{2}n^2+O(n)$.

In this paper, we determine exactly the maximum number of induced $5$-cycles in a planar graph on $n$ vertices, for $n$ sufficiently large. In order to state the formula, we define the following function.

\begin{definition}
    For $n\geq 7$, let
    $$h(n)=\max\{k_1k_2+k_2k_3+k_3k_1: k_1,k_2,k_3\in \mathbb{N}, k_1+k_2+k_3=n-4\}+2.$$
\end{definition}

Clearly, the maximum is attained when $k_1$, $k_2$ and $k_3$ are as close as possible. In particular, $h(n)=n^2/3+O(n)$.

\begin{theorem}\label{maxindc5}
	There exists a positive integer $n_0$ such that if $n\geq n_0$ and $G$ is a planar graph on $n$ vertices, then $G$ contains at most $h(n)$ induced $5$-cycles. Moreover, there exists a planar graph on $n$ vertices which contains precisely $h(n)$ induced $5$-cycles.
\end{theorem}

We first describe the extremal graph $H_n$. Since the graph has a rather complex structure, we first present a simpler $n$-vertex planar graph which has $h(n)-2$ induced $5$-cycles.

Let $S_1$, $S_2$ and $S_3$ be pairwise disjoint sets of vertices such that $|S_1|+|S_2|+|S_3|=n-4$ and $|S_1|,|S_2|,|S_3|$ are as close as possible. We define an $n$-vertex planar graph $G$ as follows. The vertex set of $G$ is the three sets of vertices $S_1$, $S_2$ and $S_3$ together with four vertices, say $w_1,w_2,w_3$ and $u$. That is, $V(G)=S_1\cup S_2\cup S_3\cup \{w_1,w_2,w_3,u\}$. We define the edges of $G$ as $E(G)=\{w_1w_2, w_2w_3,w_3w_1\}\ \cup\{w_1v,vu|\   v\in S_1\}\ \cup\{w_2v,vu| \ v\in S_2\}\ \cup\{w_3v,vu|\   v\in S_3\}$ (see Figure \ref{x=5} (a)).
It can be checked that $G$ contains exactly $|S_1||S_2|+|S_2||S_3|+|S_3||S_1|=h(n)-2$ induced $C_5$'s.

\begin{figure}
\centering
\begin{tikzpicture}[scale=0.4]
\draw[fill=black](0,2)circle(5pt);
\draw[fill=black](6,6.5)circle(1pt);
\draw[fill=black](6,7)circle(1pt);
\draw[fill=black](6,-6.5)circle(1pt);
\draw[fill=black](6,-7)circle(1pt);
\draw[fill=black](6,-1)circle(1pt);
\draw[fill=black](6,-1.5)circle(1pt);
\draw[fill=black](0,-2)circle(5pt);
\draw[fill=black](2,0)circle(5pt);
\draw[fill=black](6,8)circle(5pt);
\draw[fill=black](6,5)circle(5pt);
\draw[fill=black](6,4)circle(5pt);
\draw[fill=black](6,-4)circle(5pt);
\draw[fill=black](6,-5)circle(5pt);
\draw[fill=black](6,-8)circle(5pt);
\draw[fill=black](10,0)circle(5pt);
\draw[fill=black](6,-2)circle(5pt);
\draw[fill=black](6,0)circle(5pt);
\draw[fill=black](6,1)circle(5pt);
\draw[thick](0,2)--(0,-2)--(2,0)--(0,2);
\draw[thick](0,2)--(6,8)--(10,0)(0,2)--(6,5)--(10,0)(0,2)--(6,4)--(10,0);
\draw[thick](0,-2)--(6,-8)--(10,0)(0,-2)--(6,-5)--(10,0)(0,-2)--(6,-4)--(10,0);
\draw[thick](2,0)--(6,-2)--(10,0)(2,0)--(6,0)--(10,0)(2,0)--(6,1)--(10,0);
\draw[rotate around={90:(6,6)},red] (6,6) ellipse (3 and 1); 
\draw[rotate around={90:(6,-6)},red] (6,-6) ellipse (3 and 1); 
\draw[rotate around={90:(6,-0.5)},red] (6,-0.5) ellipse (2 and 1); 
\draw[fill=black](16,0)circle(5pt);
\draw[fill=black](24,0)circle(5pt);
\draw[fill=black](20,8)circle(5pt);
\draw[fill=black](20,-8)circle(5pt);
\draw[fill=black](20,-6)circle(5pt);
\draw[fill=black](20,-4)circle(5pt);
\draw[fill=black](20,6)circle(5pt);
\draw[fill=black](20,4)circle(5pt);
\draw[fill=black](20,2)circle(5pt);
\draw[fill=black](20,-2)circle(5pt);
\draw[fill=black](20,1)circle(2pt);
\draw[fill=black](20,0)circle(2pt);
\draw[fill=black](20,-1)circle(2pt);
\draw[thick](16,0)--(20,8)--(24,0)--(20,6)--(16,0)--(20,4)--(24,0)--(20,-8)--(16,0)--(20,-6)--(24,0)--(20,-4)--(16,0)--(20,2)--(24,0)--(20,-2)--(16,0);
\node[red] at (6, 10.2) {$S_1$};
\node[red] at (6, -10.2) {$S_3$};
\node[red] at (6, 2.2)  {$S_2$};
\node at (0,-3)  {$w_3$};
\node at (0,3)  {$w_1$};
\node at (2.5,1)  {$w_2$};
\node at (10.5,0.9)  {$u$};
\node at (6, -12) {$(a)$};
\node at (20, -12) {$(b)$};
\end{tikzpicture}
\caption{Planar graphs containing asymptotically maximum number of induced 5-cycles and 4-cycles, respectively}
\label{x=5}
\end{figure}
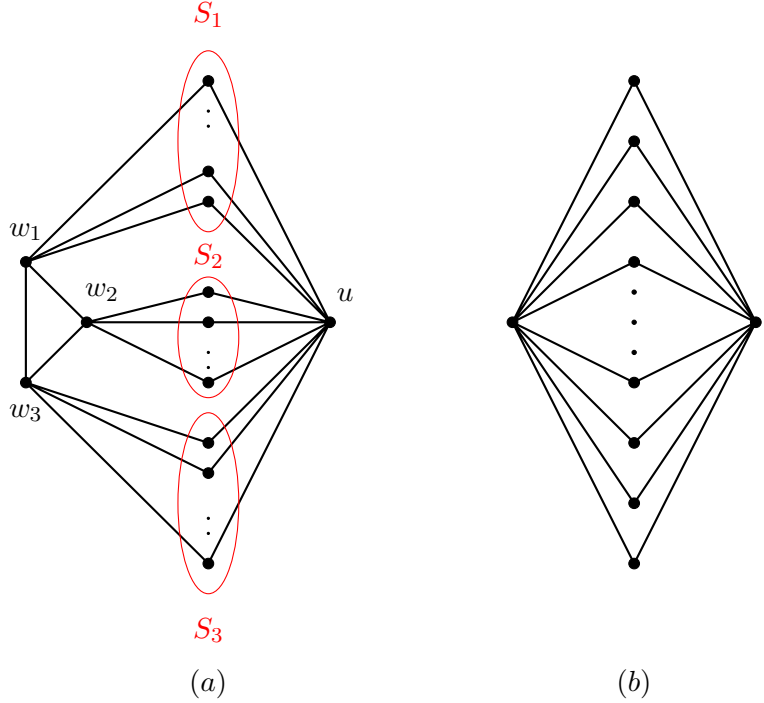

After this warm-up example, we are ready to define an $n$-vertex planar graph which has $h(n)$ induced $5$-cycles, and hence (in light of Theorem \ref{maxindc5}) is extremal for all large $n$. For an illustration, see Figure \ref{fig:extremal graph} (to see that the graph is planar, one can replace the straight dotted line between $w_1$ and $x_{3,1}$ by a non-straight edge).

\begin{definition}
    For every $n\geq 10$, define $H_n$ to be the following planar graph on $n$ vertices.
    
    The vertex set of $H_n$ is $\{w_1,w_2,w_3,w_4,u\}\cup S_1\cup S_2\cup S_3\cup \{x_{1,2},x_{2,4},x_4,x_{4,3},x_{3,1}\}$, where $|S_1|+|S_2|+|S_3|=n-10$ and $|S_1|+3$, $|S_2|$ and $|S_3|$ are as close as possible.
    
    The edges of $H_n$ are as follows:
    \begin{itemize}
        \item $w_1w_2,w_2w_3,w_3w_1,w_2w_4,w_4w_3\in E(H_n)$;
        \item for every $i\in \{1,2,3\}$ and $z\in S_i$, $w_iz,zu\in E(H_n)$;
        \item for every $(i,j)\in \{(1,2),(2,4),(4,3),(3,1)\}$, $w_ix_{i,j},w_jx_{i,j},x_{i,j}u\in E(H_n)$;
        \item $w_4x_4,x_4u,x_{2,4}x_4,x_4x_{4,3}\in E(H_n)$.
    \end{itemize}   
    
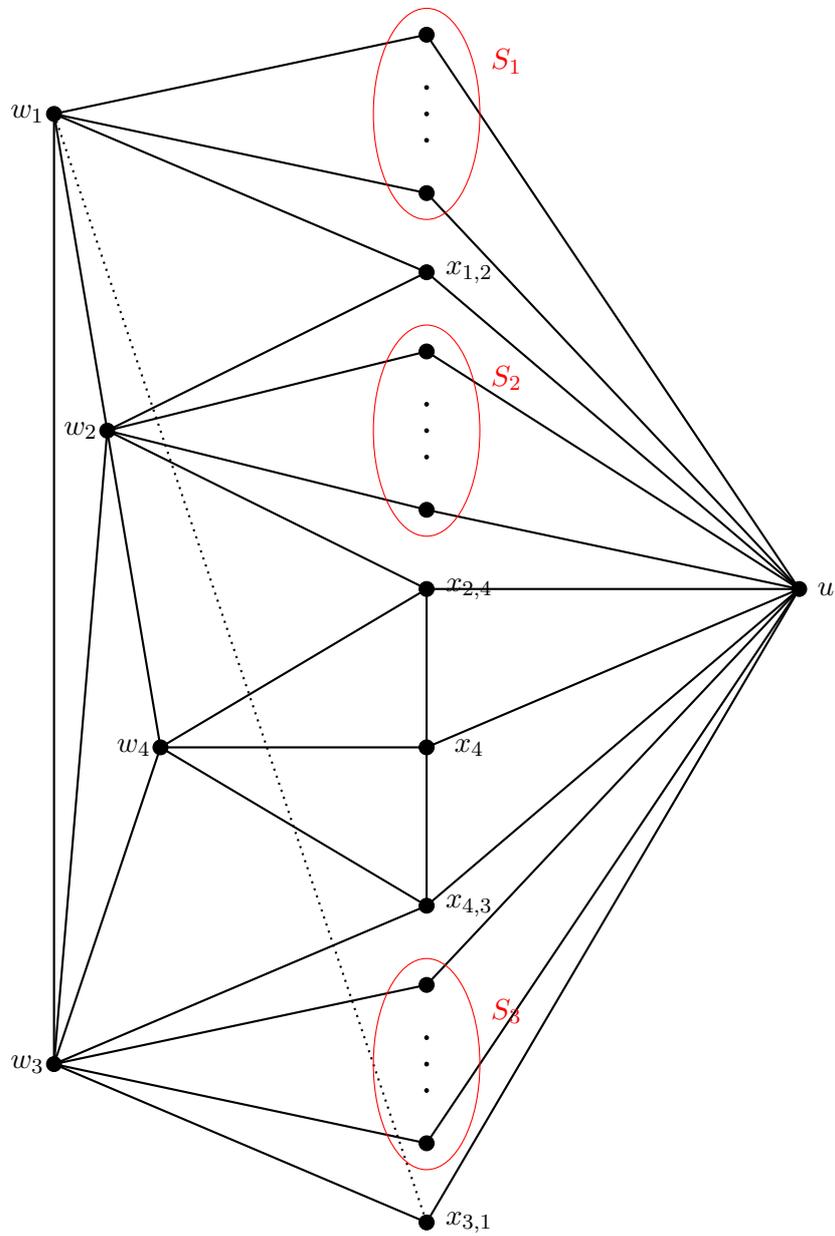
\begin{figure}
\centering
\begin{tikzpicture}[scale=0.7]
\draw[fill=black](0,-6)circle(4pt);
\draw[fill=black](2,0)circle(4pt);
\draw[fill=black](1,6)circle(4pt);
\draw[fill=black](14,3)circle(4pt);
\draw[fill=black](0,12)circle(4pt);
\draw[fill=black](7,9)circle(4pt);
\draw[fill=black](7,3)circle(4pt);
\draw[fill=black](7,-3)circle(4pt);
\draw[fill=black](7,0)circle(4pt);
\draw[fill=black](7,-9)circle(4pt);
\draw[fill=black](7,10.5)circle(4pt);
\draw[fill=black](7,13.5)circle(4pt);
\draw[fill=black](7,4.5)circle(4pt);
\draw[fill=black](7,7.5)circle(4pt);
\draw[fill=black](7,-7.5)circle(4pt);
\draw[fill=black](7,-4.5)circle(4pt);
\draw[fill=black](7,12)circle(1pt);
\draw[fill=black](7,12.5)circle(1pt);
\draw[fill=black](7,11.5)circle(1pt);
\draw[fill=black](7,6)circle(1pt);
\draw[fill=black](7,6.5)circle(1pt);
\draw[fill=black](7,5.5)circle(1pt);
\draw[fill=black](7,-6)circle(1pt);
\draw[fill=black](7,-6.5)circle(1pt);
\draw[fill=black](7,-5.5)circle(1pt);

\draw[thick](0,12)--(1,6)--(0,-6)--(2,0)--(1,6);
\draw[thick](0,12)--(0,-6);

\draw[thick](0,12)--(7,13.5)--(14,3);
\draw[thick](1,6)--(7,7.5)--(14,3);
\draw[thick](0,-6)--(7,-4.5)--(14,3);
\draw[thick](0,12)--(7,10.5)--(14,3);
\draw[thick](1,6)--(7,4.5)--(14,3);
\draw[thick](0,-6)--(7,-7.5)--(14,3);
\draw[thick](0,12)--(7,9)--(14,3);
\draw[thick](1,6)--(7,3)--(14,3);
\draw[thick](0,-6)--(7,-9)--(14,3);
\draw[thick](1,6)--(7,9);
\draw[thick](2,0)--(7,3);
\draw[thick](2,0)--(7,0);
\draw[thick](2,0)--(7,-3);
\draw[thick](7,3)--(7,0)--(7,-3)--(0,-6);
\draw[thick](7,0)--(14,3)--(7,-3);
\draw[dotted,thick](0,12)--(7,-9);

\draw[rotate around={90:(7,12)},red] (7,12) ellipse (2 and 1); 
\draw[rotate around={90:(7,6)},red] (7,6) ellipse (2 and 1);
\draw[rotate around={90:(7,-6)},red] (7,-6) ellipse (2 and 1);
\node[red] at (8.5, 13) {$S_1$};
\node[red] at (8.5, -5) {$S_3$};
\node[red] at (8.5, 7)  {$S_2$};
\node at (7.8,9) {$x_{1,2}$};
\node at (7.8,3) {$x_{2,4}$};
\node at (7.8,0) {$x_4$};
\node at (7.8,-3) {$x_{4,3}$};
\node at (7.8,-9) {$x_{3,1}$};
\node at (-0.5,12) {$w_1$};
\node at (-0.5,-6)  {$w_3$};
\node at (0.5,6)  {$w_2$};
\node at (1.5,0)  {$w_4$};
\node at (14.5,3)  {$u$};
\end{tikzpicture}
\caption{The planar graph $H_n$ containing the maximum number of induced 5-cycles}
\label{fig:extremal graph}
\end{figure}
\end{definition}

After a somewhat tedious check, one can see that when $n$ is sufficiently large, then the only induced $C_5$ in $H_n$ which does not contain $u$ is $w_2w_3x_{4,3}x_4x_{2,4}$. Moreover, the induced $C_5$'s containing $u$ also contain precisely two of the $w_i$'s. The number of induced $C_5$'s containing $u$, $w_1$ and $w_2$ is $(|S_1|+1)(|S_2|+1)$. Indeed, the common neighbour of $u$ and $w_1$ in such an induced $C_5$ can be any vertex from $S_1\cup \{x_{3,1}\}$ (but not $x_{1,2}$ because it is a neighbour of $w_2$), and the common neighbour of $u$ and $w_2$ can be any vertex from $S_2\cup \{x_{2,4}\}$. Similarly, the number of induced $C_5$'s containing $u$, $w_1$ and $w_3$ is $(|S_1|+1)(|S_3|+1)$, the number of induced $C_5$'s containing $u$, $w_2$ and $w_3$ is $(|S_2|+2)(|S_3|+2)$, the number of induced $C_5$'s containing $u$, $w_2$ and $w_4$ is $2(|S_2|+1)$, the number of induced $C_5$'s containing $u$, $w_3$ and $w_4$ is $2(|S_3|+1)$ and there is no $C_5$ containing $u$, $w_1$ and $w_4$. Altogether, we find that the number of induced $C_5$'s in $H_n$ is
\begin{align*}
    &1+(|S_1|+1)(|S_2|+1)+(|S_1|+1)(|S_3|+1)+(|S_2|+2)(|S_3|+2)+2(|S_2|+1)+2(|S_3|+1) \\
    &= (|S_1|+4)(|S_2|+1)+(|S_1|+4)(|S_3|+1)+(|S_2|+1)(|S_3|+1)+2.
\end{align*}
Since $(|S_1|+4)+(|S_2|+1)+(|S_3|+1)=n-4$ and $|S_1|+4,|S_2|+1,|S_3|+1$ are as close as possible, the above expression is equal to $h(n)$.

It remains to prove the upper bound in Theorem \ref{maxindc5}. This is done by an induction argument. The key result is the following.

\begin{proposition} \label{prop:combine all}
    There exists a positive integer $n_1$ such that the following holds. Let $n\geq n_1$ and let $G$ be an $n$-vertex planar graph which has $h(n)+t$ induced $5$-cycles for some $t\geq 1$. Then either $G$ has a subgraph on $n-1$ vertices which has at least $h(n-1)+t+1$ induced $5$-cycles, or $G$ has a subgraph on $n-3$ vertices which has at least $h(n-3)+t+1$ induced $5$-cycles.
\end{proposition}

Let us see how this proposition implies our main result.

\begin{proof}[Proof of Theorem \ref{maxindc5}]
We have already given a construction which has $h(n)$ induced $5$-cycles, so it suffices to prove the upper bound. Let $n_1$ be the positive integer provided by Proposition~\ref{prop:combine all}. Define $n_0=n_1+3n_1^5$ and let $G$ be a planar graph on $n\geq n_0$ vertices. Assume, for contradiction, that $G$ contains more than $h(n)$ induced $5$-cycles. Let the number of $5$-cycles in $G$ be $h(n)+t$ for some $t\geq 1$. By Proposition \ref{prop:combine all}, $G$ has a subgraph $G'$ on $n'\geq n-3$ vertices which has at least $h(n')+t+1$ induced $5$-cycles. Again by Proposition \ref{prop:combine all}, $G'$ has a subgraph $G''$ on $n''\geq n'-3\geq n-6$ vertices which has at least $h(n'')+t+2$ induced $5$-cycles. Repeat this as long as the subgraph has at least $n_1$ vertices. Eventually, we are left with a subgraph $G_{\textrm{final}}$ on $n_{\textrm{final}}<n_1$ vertices which has at least $h(n_{\textrm{final}})+t+n_1^5$ induced $5$-cycles. This is clearly a contradiction.
\end{proof}

The rest of the paper is devoted to proving Proposition \ref{prop:combine all}. Note that if there exists a vertex which is contained in fewer than $h(n)-h(n-1)$ induced $5$-cycles, then we can remove this vertex and we are done. Similarly, if there are three vertices such that their removal deletes fewer than $h(n)-h(n-3)$ induced $5$-cycles, then we can remove these vertices and we are done. Hence, in what follows, we can assume that such vertices do not exist. We will use these assumptions to find more and more structure in our graph. Eventually, the structure of the remaining possibilities will be so restricted that we can directly bound the number of induced $5$-cycles, and thereby reach a contradiction.

We will see by straightforward computations that $h(n)-h(n-1)>2n/3-4$ and $h(n)-h(n-3)=2n-11$.

In Section \ref{sec:find k27}, we prove that if (a drawing of) our graph does not contain a $K_{2,7}$ which is ``empty", i.e. which has no other vertices inside, then there is a vertex which is contained in at most $11n/20$ induced $5$-cycles. Since $11n/20$ is less than $h(n)-h(n-1)$ for large $n$, this means that we can assume that our graph does contain an empty $K_{2,7}$.

In Section \ref{sec:one vx removal}, we reveal even more structure in our graph by using that there is no vertex which is contained in less than $2n/3-10$ induced $5$-cycles (note that $2n/3-10<h(n)-h(n-1)$). For example, Lemma~\ref{lem:three empty k27s} provides us with a structure that already starts to resemble the near-extremal graph depicted in Figure \ref{x=5}.

Since the value of $h(n)-h(n-1)$ depends on the remainder of $n$ modulo $3$, it is not convenient to remove just one vertex when we are already very close to the extremal example. Instead, in Section \ref{sec:three vx removal}, we carefully choose three vertices whose removal does not decrease the number of induced $5$-cycles by too much. More precisely, in Lemma \ref{lem:two scenarios}, we prove that in a graph which has all the properties that we have already obtained by the earlier lemmas, either there are three vertices whose removal deletes fewer than $h(n)-h(n-3)$ induced $5$-cycles, or the graph has a very specific structure (and is very similar to the extremal graph $H_n$).

In Section \ref{sec:finishing}, we show that if the graph has the structure forced by Lemma \ref{lem:two scenarios}, then it has at most $h(n)$ induced $5$-cycles, completing the proof of Proposition \ref{prop:combine all}.

\section{Preliminaries}
Let $G$ be a planar graph on $n$ vertices.   We denote the edge and vertex sets of $G$ with $V(G)$ and $E(G)$, respectively.   For a vertex $v\in V(G)$, we denote its degree by $d_G(v)$.   The set of its neighbours is denoted as $N_G(v)$.   For simplicity, we drop the subscript whenever the graph is clear.

We start with a basic lemma, which we are going to use throughout the paper.
\begin{lemma} \label{basicbound}
	Let $G$ be a planar graph, let $v\in V(G)$, and let $u$ and $w$ be distinct neighbours of $v$.   Let $X_0=N(u)\setminus (N(w)\cup \{w\})$ and let $Y_0=N(w)\setminus (N(u)\cup \{u\})$.   Let $X$ be the subset of $X_0$ consisting of those vertices that have at least one neighbour in $Y_0$, and let $Y$ be the subset of $Y_0$ consisting of those vertices that have at least one neighbour in $X_0$. Let $\mu$ be the number of connected components in the induced bipartite subgraph of $G$ with parts $X$ and $Y$. Then the number of induced $C_5$'s in $G$ containing $u$, $v$ and $w$ is at most $|X|+|Y|-\mu$. In particular, it is always at most $|X|+|Y|-1$.
\end{lemma}

\begin{proof}
	Clearly any such $C_5$ contains precisely one vertex from each of $X$ and $Y$.   Hence, the number of such induced $C_5$'s is at most the number of edges between $X$ and $Y$.
	However, the induced bipartite subgraph of $G$ with parts $X$ and $Y$ is acyclic.   Indeed, suppose that there is a cycle $x_1y_1x_2y_2\dots x_ky_kx_1$ with $x_i\in X$ and for all $i$ and $y_j\in Y$ for all $j$.   The subgraph of $G$ with vertices $u,v,w,x_1,y_1,\dots,x_k,y_k$ and edges $uv,vw,ux_1,ux_2,wy_1,wy_2,x_1y_1,y_1x_2,x_2y_2,y_2x_3,\dots,y_kx_1$ is a subdivision of $K_{3,3}$ with the parts being $\{u,y_1,y_2\}$ and $\{w,x_1,x_2\}$. Indeed, the only edge of this $K_{3,3}$ which is potentially not present in $G$ is $x_1y_2$, but we have a path $y_2x_3y_3\dots x_ky_kx_1$ in $G$. Hence, $G$ is not planar by (the easier direction of) Kuratowski’s theorem \cite{kuratowski1930probleme},  which is a contradiction.   Thus, the induced bipartite subgraph of $G$ with parts $X$ and $Y$ is a forest. The statement follows.
\end{proof}

\section{Finding an empty $K_{2,7}$} \label{sec:find k27}

In this section we prove that if $G$ does not contain an empty $K_{2,7}$, then there is even a vertex which is contained in at most $11n/20$ induced $C_5$'s. Here an empty $K_{2,7}$ in a drawing of $G$ means distinct vertices $u$ and $w$, and $z_1,\dots,z_7\in N(u)\cap N(w)$ in natural order such that the bounded region with boundary consisting of $uz_1$, $z_1w$, $wz_7$ and $z_7u$ contains no vertex other than $z_2,\dots,z_6$.

\begin{lemma} \label{nok27}
	Let $n$ be sufficiently large and let $G$ be a plane graph on $n$ vertices.   If $G$ does not contain an empty (not necessary induced) $K_{2,7}$, then there is a vertex in $G$ which is contained in at most $11n/20$ induced $C_5$'s.
\end{lemma}

To prove this, we need some preliminaries.

\begin{lemma} \label{findk2m}
	Let $n$ be sufficiently large and let $G$ be a planar graph on $n$ vertices.   If $G$ does not contain distinct vertices $u$ and $w$ with $|N(u)\cap N(w)|\geq n/10^6$, then there is a vertex in $G$ which is contained in at most $n/2$ induced $C_5$'s.
\end{lemma}

\begin{proof}
	Suppose otherwise. Let $v$ be a vertex of degree at most $5$ in $G$. Such vertex exists since the average degree in a planar graph is strictly less than $6$ by Euler's formula.   Then by the pigeonhole principle $v$ has distinct non-adjacent neighbours $u$ and $w$ such that the number of induced $C_5$'s containing $u$, $v$ and $w$ is at least $n/20$.   Define $X$ and $Y$ as in Lemma \ref{basicbound}.   By the same lemma, we have $|X|+|Y|\geq n/20$.   Let $G'$ be the induced bipartite subgraph of $G$ with parts $X$ and $Y$.   By assumption, there is no vertex of degree at least $n/10^6$ in $G'$.   Then since $G'$ has at least $\frac{|X|+|Y|}{2}\geq n/40$ edges, there must exist a set of at least $10^4$ independent edges in $G'$.
	
	Let them be $x_1y_1,x_2y_2,\dots,x_{10^4}y_{10^4}$ such that $x_1,x_2,\dots,x_{10^4}\in X$, the edges $ux_1,ux_2,\dots,ux_{10^4}$ are in anti-clockwise order, and the bounded region with boundary consisting of edges $ux_1$, $x_1y_1,y_1w,wy_{10^4},y_{10^4}x_{10^4},x_{10^4}u$ contains all $x_i$ and $y_i$. For $1\leq i\leq 10^4-1$, let $R_i$ be the bounded region with boundary consisting of $ux_i,x_iy_i,y_iw,wy_{i+1},y_{i+1}x_{i+1},x_{i+1}u$. Choose $11\leq i\leq 10^4-12$ such that the number of vertices in $R_{i-10}\cup R_{i-9}\cup \dots \cup R_{i+11}$ is at most $n/300$. Let $R=R_i\cup R_{i+1}$.
	
	Let $S$ be the set of vertices of $G$ in the interior of $R$ which do not belong to $N(u)\cap N(w)$. Note that $x_{i+1}\in S$, so $S\neq \emptyset$. Now the graph $G''=G\lbrack S\rbrack$ is planar, so there exists some $z\in S$ which has degree at most $5$ in $G''$. But it is joined to at most $2$ elements of $N(u)\cap N(w)$, so it has at most $7$ neighbours in the interior of $R$. Hence (together with $u,x_i,y_i,w,y_{i+2}$ and $x_{i+2}$), $z$ has at most $13$ neighbours.
	
By assumption, $z$ is contained in at least $n/2$ induced $C_5$'s. We claim that any such $C_5$ is either contained entirely in $R_{i-10}\cup R_{i-9}\cup \dots \cup R_{i+11}$ or it contains both $u$ and $w$. Indeed, let $C$ be an induced $5$-cycle containing $z$ which leaves the region $R_{i-10}\cup R_{i-9} \dots \cup R_{i+11}$ and let $q$ be a vertex of $C$ outside $R_{i-10}\cup R_{i-9}\cup \dots \cup R_{i+11}$. Then $C$ in the union of two internally vertex-disjoint paths of length at most $5$ between $z$ and $q$. However, since $z$ is inside $R_i\cup R_{i+1}$ and $q$ is outside $R_{i-10}\cup R_{i-9}\cup \dots \cup R_{i+11}$, any such path must pass through either $u$ or $w$. Hence, $C$ contains both $u$ and $w$.

If an induced $5$-cycle is contained in $R_{i-10}\cup R_{i-9}\cup \dots \cup R_{i+11}$, then it can only use a set of at most $n/300$ vertices, and since $z$ has degree at most $13$, by Lemma \ref{basicbound} there are at most ${13 \choose 2}\cdot n/300<n/3$ such induced $C_5$'s.   So there are at least $n/6$ induced $C_5$'s containing $z$, $u$ and $w$. Recall that $u$ and $w$ are non-adjacent and $z\not \in N(u)\cap N(w)$.  If $z\in N(u)$, then all these induced $C_5$'s are of the form $uzswt$ for some $s\in N(z)$ and $t\in N(u)\cap N(w)$, while if $z\in N(w)$, then all these induced $C_5$'s are of the form $uszwt$ for some $s\in N(z)$ and $t\in N(u)\cap N(w)$.   In either case, since $|N(z)|\leq 13$, it follows that $|N(u)\cap N(w)|\geq \frac{n}{6\cdot 13}> \frac{n}{10^6}$.   This contradicts the condition in the lemma.
\end{proof}
\begin{lemma} \label{emptyorfull}
	Let $n$ be sufficiently large and let $G$ be a plane graph on $n$ vertices. Let $u$ and $w$ be distinct vertices, and let $v_1,v_2,\dots,v_6$ be some of their common neighbours, in natural order. Assume that the number of vertices in the interior of the bounded region with boundary consisting of $uv_3$, $v_3w$, $wv_4$ and $v_4u$ is at least one but at most $n^{1/5}$ and that there is no common neighbour of $u$ and $w$ in the same region. Then $G$ has a vertex which is contained in at most $11n/20$ induced $C_5$'s.
\end{lemma}

\begin{proof}
	Suppose otherwise. Let $R$ be the bounded region with boundary consisting of $uv_3$, $v_3w$, $wv_4$ and $v_4u$. Let $x$ be an arbitrary vertex inside $R$. By assumption, $x\not \in N(u)\cap N(w)$. Since there are at most $n^{1/5}+4$ vertices in $R$ (including its boundary), the number of induced $C_5$'s containing $x$ which lie entirely in $R$ (possibly touching the boundary) is at most $(n^{1/5}+4)^4\leq n/20$. Thus, since $x$ is contained in at least $11n/20$ induced $C_5$'s, there exist at least $n/2$ induced $C_5$'s containing $x$ which contain vertices outside $R$.
	
	Take such an induced $C_5$ and call it $C$. We claim that $C$ must contain both $u$ and $w$, but does not contain $v_3$ and $v_4$. Indeed, if we go through the vertices of $C$ one by one in natural order, starting with $x$, then there will be a vertex from the set $\{u,v_3,w,v_4\}$ right before the walk first leaves $R$, and then one in the same set when the walk first returns to $R$. Call these two vertices $y$ and $z$, respectively. Since $C$ contains the vertex $x$, which is in the interior of $R$, it follows that $y$ and $z$ are not neighbours in $C$, so they are also not neighbours in $G$. Thus, either $\{y,z\}=\{u,w\}$ or $\{y,z\}=\{v_3,v_4\}$. In the latter case, again since $C$ is induced and contains $x$, $C$ contains neither $u$ nor $w$. So there exists a path of length at most $3$ in $C$, and therefore also in $G$, from $v_3$ to $v_4$ outside of $R$ which avoids both $u$ and $w$. This is clearly not possible because of the vertices $v_1$, $v_2$, $v_5$ and $v_6$.
	
	Thus, $C$ indeed contains both $u$ and $w$, and it is easy to see that it does not contain $v_3$ and $v_4$. Since $x\not \in N(u)\cap N(w)$, it follows that either $x\in N(u)$ and $C=uxqwr$ for some $q\in N(x)\cap N(w)\setminus \{v_3,v_4\}$ and $r\in N(u)\cap N(w)$, or $x\in N(w)$ and $C=uqxwr$ for some $q\in N(x)\cap N(u)\setminus \{v_3,v_4\}$ and $r\in N(u)\cap N(w)$. In particular, it follows that $N(u)$ and $N(w)$ both have vertices in the interior of $R$.
	
	Let $X$ be the set of vertices of $N(u)$ in the interior of $R$ and let $Y$ be the set of vertices of $N(w)$ in the interior of $R$. Similarly as in the proof of Lemma \ref{basicbound}, the induced bipartite subgraph of $G$ with parts $X$ and $Y$ is acyclic. Thus, there is a vertex in that graph of degree at most one. Without loss of generality, we may assume that some $x\in X$ has at most one neighbour in $Y$. Then, by the previous paragraph, there are at most $|N(u)\cap N(w)|$ induced $C_5$'s containing $x$ as well as vertices outside $R$. Thus, by the first paragraph, $|N(u)\cap N(w)|\geq n/2$.
    
    By a simple averaging, it follows that there exist distinct $t_1,t_2,\dots,t_7\in N(u)\cap N(w)$ (in natural order) such that the region $S$ bounded by $ut_1,t_1w,wt_7,t_7u$ contains at most $100$ vertices.  Now any induced $C_5$ which contains $t_4$ and has vertices outside $S$ must contain $u$ and $w$.  Such an induced $C_5$ cannot contain any vertices from $N(u)\cap N(w)$ other than $t_4$, so by Lemma \ref{basicbound}, there are at most $n/2$ such induced $C_5$'s.  The number of induced $C_5$'s containing $t_4$ but no vertices outside $S$ is at most $100^5$, so $t_4$ satisfies the conclusion of the lemma.
\end{proof}

\begin{corollary} \label{emptyk27}
	Let $n$ be sufficiently large and let $G$ be a plane graph on $n$ vertices. Let $u$ and $w$ be distinct vertices with $|N(u)\cap N(w)|\geq 7n^{4/5}$. Then either $G$ has an empty $K_{2,7}$ whose part of size two is $\{u,w\}$ or there is a vertex in $G$ which is contained in at most $11n/20$ induced $C_5$'s.
\end{corollary}

\begin{proof}
	Let $v_1,v_2,\dots,v_{t}$ be the common neighbours of $u$ and $w$ in natural order.  For each $1\leq i\leq t-1$, let $R_i$ be the bounded region with boundary consisting of the edges $uv_i$, $v_iw$, $wv_{i+1}$ and $v_{i+1}u$.  By Lemma \ref{emptyorfull}, each $R_i$ with $3\leq i\leq t-3$ contains either zero or at least $n^{1/5}$ vertices in its interior.  Hence, the number of non-empty $R_i$'s is at most $n^{4/5}+4$. Since $t\geq 7n^{4/5}$, there exists some $1\leq i \leq t-6$ for which $u,w,v_i,v_{i+1},\dots,v_{i+6}$ define an empty $K_{2,7}$.
\end{proof}

\medskip

Now Lemma \ref{nok27} follows from Lemma \ref{findk2m} and Corollary \ref{emptyk27}.

\section{The rough structure of near-extremal graphs} \label{sec:one vx removal}

\begin{lemma} \label{lem:compare n and n-1}
    Let $n\geq 8$. Then $h(n)-h(n-1)>2n/3-4$.
\end{lemma}

\begin{proof}
Choose $k_1,k_2,k_3\in \mathbb{N}$ such that $k_1+k_2+k_3=n-4$, $k_1,k_2,k_3$ are as close as possible and $k_1\geq k_2\geq k_3$. Then $h(n)=k_1k_2+k_2k_3+k_3k_1+2$ and $h(n-1)=(k_1-1)k_2+k_2k_3+k_3(k_1-1)+2$, so $h(n)-h(n-1)=k_2+k_3\geq \frac{2}{3}(k_2+k_3+(k_1-1))=\frac{2}{3}(n-5)>2n/3-4$.
\end{proof}

Lemma \ref{lem:compare n and n-1} means that we are happy (when proving Proposition \ref{prop:combine all}) if we can find a vertex which is contained in at most, say, $2n/3-10$ induced $5$-cycles. The next lemma guarantees some structure if such a vertex does not exist.

\begin{lemma} \label{structureofexc}
	Let $n$ be sufficiently large and let $G$ be a plane graph on $n$ vertices. Suppose that $G$ has at least $\frac{5}{18}n^2$ induced $C_5$'s and that there does not exist a vertex in $G$ which is contained in at most $2n/3-10$ induced $C_5$'s. Let $u$ and $w$ be the two vertices in the part of size $2$ of an empty $K_{2,7}$. Then $u$ and $w$ are non-adjacent and there exist sets $X\s N(u)\setminus N(w)$ and $Y\s N(w)\setminus N(u)$ with the following properties.
	
	\begin{enumerate}
		\item $|X|+|Y|\geq 2n/3-10$, every $x\in X$ is adjacent to at least one element of $Y$ and every $y\in Y$ is adjacent to at least one element of $X$.
		\item The bipartite induced subgraph of $G$ with parts $X$ and $Y$ has maximum degree at least $n^{5/6}$. \label{prop:maxdeg}
	\end{enumerate}
\end{lemma}

\begin{proof}
	Let $v$ be the centre vertex in the part of size $7$ in the empty $K_{2,7}$. Define $X$ and $Y$ as in the statement of Lemma \ref{basicbound}. Since every induced $C_5$ containing $v$ also contains $u$ and $w$, and by assumption $v$ is contained in more than $2n/3-10$ induced $C_5$'s, it follows by Lemma \ref{basicbound} that $|X|+|Y|>2n/3-10$. Moreover, since there exists an induced $C_5$ containing $u$, $v$ and $w$, it follows that $u$ and $w$ are non-adjacent.
	
	Let $G'$ be the induced bipartite subgraph of $G$ with parts $X$ and $Y$. It remains to show that $G'$ has maximum degree at least $n^{5/6}$. Suppose otherwise.
	
	Let $x\in X$ be an arbitrary vertex.  We give an estimate for the number of induced $C_5$'s containing $x$.  We first count those $C_5$'s which contain both $u$ and $w$ as vertices.  Let us call these \emph{type 1} $C_5$'s.  Since $w$ is non-adjacent to both $x$ and $u$, the number of type 1 $C_5$'s containing $x$ is at most $d_{G'}(x)\cdot t$, where $d_{G'}(x)$ is the degree of $x$ in $G'$ and $t=|N(u)\cap N(w)|$.
	
	Call those induced $C_5$'s which do not contain both $u$ and $w$ \emph{type 2}.  
	To bound the number of such $C_5$'s, we will use the following claim.
	
	\medskip
	
	\begin{claim}\label{Claim2}
	For every $q\in V(G)$, the number of vertices $z\in X\cup Y$ for which there exists a path of length at most $3$ between $q$ and $z$ avoiding both $u$ and $w$ is at most $100n^{5/6}$.
	\end{claim}
	
	\begin{proof}
	    Take a maximal matching between $X$ and $Y$.  Let the edges in this matching be $x_{i_1}y_{i_1},\dots,x_{i_s}y_{i_s}$ such that $x_{i_j}\in X, y_{i_j}\in Y$ and the edges $wy_{i_1},\dots,wy_{i_s}$ are in clockwise order.  For each $1\leq j\leq s-1$, let $R_j$ be the bounded region with boundary consisting of the edges $ux_{i_j},x_{i_j}y_{i_j},y_{i_j}w,wy_{i_{j+1}},y_{i_{j+1}}x_{i_{j+1}},x_{i_{j+1}}u$, and let $R_0$ be the unbounded region with boundary consisting of the edges $ux_{i_1},x_{i_1}y_{i_1},y_{i_1}w,wy_{i_s},y_{i_s}x_{i_s},x_{i_s}u$.  Let $0\leq j\leq s-1$.  By the maximality of our matching, any element of $X\cup Y$ in the interior of $R_j$ is a neighbour in $G'$ of some vertex in $X\cup Y$ on the boundary of $R_j$.  Since there are $4$ vertices in $X\cup Y$ on the boundary of $R_j$, and $G'$ has maximum degree less than $n^{5/6}$, there are at most $4n^{5/6}$ elements of $X\cup Y$ in the interior of $R_j$.
	
	Let $q\in V(G)\setminus \{u,w\}$.  Then $q$ is in $R_j$ (possibly on the boundary) for some $0\leq j\leq s-1$.  If there exists some $z\in X\cup Y$ for which there is a path of length at most $3$ from $q$ to $z$ avoiding both $u$ and $w$, then $z$ is in $R_{j-4}\cup R_{j-3}\cup \dots R_{j+4}$ (with the subscripts considered modulo $s$).  But there are at most $9\cdot 4n^{5/6}$ such vertices $z$, which finishes the proof of the claim.
	\end{proof}

	Recall that $G'$ is acyclic, so the number of edges in $G'$ is at most $|X|+|Y|-1$.  Thus, if $\ell$ is the number of vertices of degree at least $3$ in $G'$, then $3\ell\leq 2(|X|+|Y|)$, so the number of vertices of degree at most $2$ in $G'$ is $|X|+|Y|-\ell\geq \frac{|X|+|Y|}{3}\geq \frac{2n}{9}-4$.
	
	The number of edges in $G$ is at most $3n$ by Euler's formula, so the number of vertices in $G$ of degree at least $60$ is at most $n/10$.
	
	Moreover, it follows from Claim \ref{Claim2} by double counting that the number of vertices $z\in X\cup Y$ for which there exist at least $1000n^{5/6}$ vertices $q\in V(G)$ with a path of length at most $3$ between $z$ and $q$ and avoiding both $u$ and $w$ is at most $n/10$.
	
	Thus, there exists a vertex $z\in X\cup Y$ which has degree at most $2$ in $G'$, degree at most $60$ in $G$ and for which the number of $q\in V(G)$ with a path of length at most $3$ between $z$ and $q$ avoiding $u$ and $w$ is at most $1000n^{5/6}$.

	Suppose that $q\in V(G)$ is distinct from $z$, $u$ and $w$, and that there exists a type 2 induced $C_5$ containing both $z$ and $q$. Then there exists a path of length at most $3$ from $q$ to $z$ which contains neither $u$ nor $w$. But there are at most $1000n^{5/6}$ such vertices $q\in V(G)$, so by Lemma~\ref{basicbound}, the number of type 2 induced $C_5$'s containing $z$ is at most ${60 \choose 2}\cdot (1000n^{5/6}+2)$.  Moreover, the number of type 1 induced $C_5$'s containing $z$ is at most $2t$, where $t=|N(u)\cap N(w)|$.  Since the total number of induced $C_5$'s containing $z$ is assumed to be at least $2n/3-10$, it follows that $|N(u)\cap N(w)|\geq n/3-n^{6/7}$.
	
	\begin{claim} \label{claim:2/9}
	    The number of induced $C_5$'s in $G$ is at most $(\frac{2}{9}+o(1))n^2$.
	\end{claim}
	
	\begin{proof}
	Let $N(u)\cap N(w)=\{v_1,\dots,v_t\}$ such that $uv_1,uv_2,\dots,uv_t$ are in anticlockwise order and the bounded region with boundary consisting of $uv_1,v_1w,wv_t,v_tu$ contains all the $v_i$'s.  For $1\leq i\leq t-1$, let $T_i$ be the bounded region with boundary consisting of $uv_i,v_iw,wv_{i+1},v_{i+1}u$. Suppose that there are at least $n^{4/5}+4$ values of $i$ for which the interior of $T_i$ contains a vertex of $G$. Then there is some $3\leq i\leq t-3$ such that the number of vertices in the interior of $T_i$ is at least one but at most $n^{1/5}$. Then, by Lemma \ref{emptyorfull}, there is a vertex in $G$ that is contained in at most $11n/20$ induced $C_5$'s, which is a contradiction.  Thus, for all but $o(n)$ choices $6\leq i\leq t-6$ the regions $T_{i-5},T_{i-4},\dots,T_{i+5}$ contain no vertex in their interior.  But for all such $i$, we have that $v_i$ is contained in at most $2n/3+o(n)$ induced $C_5$'s.
	
	Let us remove the vertices $v_i$ for these values of $i$ from $G$ and note that with this we remove at least $n/3-o(n)$ vertices but at most $(\frac{2}{9}+o(1))n^2$ induced $C_5$'s (since we have $n/3-n^{6/7}\leq |N(u)\cap N(w)|\leq n/3+10$).  It suffices to show that in the remaining graph $\tilde{G}$ there are at most $o(n^2)$ induced $C_5$'s.  Let $S=V(\tilde{G})\setminus (X\cup Y\cup \{u,w\})$.  Note that $|S|=o(n)$.
	
	Now we remove the vertices in $S$ one by one in careful order, such that in each step we remove $O(n)$ induced $C_5$'s.  Note that any $v\in S$ is joined to at most $2$ vertices from $X$ (else $G$ contains a subdivision of $K_{3,3}$ which is impossible by Kuratowski’s theorem). Similarly, it is joined to at most $2$ vertices from $Y$, so it is joined to at most $6$ vertices from $X\cup Y\cup \{u,w\}$. Thus, since $\tilde{G}$ is planar, we may remove the vertices of $S$ one by one in a way that in each step the removed vertex has at most $11$ neighbours in the current graph.  This way, by Lemma \ref{basicbound}, we remove at most ${11 \choose 2}\cdot n$ induced $C_5$'s in each step.  Thus, while removing the vertices in $S$, we remove at most $o(n^2)$ induced $C_5$'s.
	
	It remains to prove that in $G\lbrack X\cup Y\cup \{u,w\}\rbrack$ there are $o(n^2)$ induced $C_5$'s.  To show this, we prove that we may remove the vertices in $X\cup Y$ one by one such that in each step we remove $o(n)$ induced $C_5$'s.  Clearly, in each step we can remove a vertex $q\in X\cup Y$ which has degree at most $6$ in the current graph.  We claim that $q$ is then contained in at most $o(n)$ induced $C_5$'s.  Let $Z$ be the set of vertices $z\in X\cup Y$ for which there is a path of length at most $3$ from $q$ to $z$ which avoids both $u$ and $w$.  By Claim \ref{Claim2}, we have $|Z|=o(n)$.  Since $N(u)\cap N(w)\cap (X\cup Y)=\emptyset$, there is no induced $C_5$ with vertices from $X\cup Y\cup \{u,w\}$ which contains both $u$ and $w$, so any induced $C_5$ which contains $q$ must consist of vertices from the set $Z\cup \{u,w\}$.  Thus, as $q$ has degree at most $6$, by Lemma \ref{basicbound} there are at most $o(n)$ induced $C_5$'s containing $q$.     
	\end{proof}
	
	Claim \ref{claim:2/9} contradicts our hypothesis that $G$ has at least $\frac{5}{18}n^2$ induced $C_5$'s, so the proof is complete.
\end{proof}

\begin{lemma} \label{lem:three empty k27s}
    Let $n$ be sufficiently large and let $G$ be a plane graph on $n$ vertices. Suppose that $G$ has at least $\frac{5}{18}n^2$ induced $C_5$'s and that there does not exist a vertex in $G$ which is contained in at most $2n/3-10$ induced $C_5$'s. Then there exist distinct vertices $w_1$, $w_2$, $w_3$ and $u$ such that $w_1w_2,w_2w_3,w_3w_1\in E(G)$ and for every $1\leq i\leq 3$, there exists an empty $K_{2,7}$ whose part of size two is $\{w_i,u\}$.
\end{lemma}

\begin{proof}
By Lemma \ref{nok27}, $G$ contains an empty $K_{2,7}$. Let $u$ and $w$ be the two vertices in the part of size two. Let $X$ and $Y$ be the sets provided by Lemma \ref{structureofexc}. By Property \ref{prop:maxdeg} from that lemma, we may assume without loss of generality (after swapping $u$ and $w$ if necessary) that some $y\in Y$ has at least $n^{5/6}$ neighbours in $X$.  Let $v$ be an arbitrary common neighbour of $u$ and $w$ and order the elements of $Y$ as $y_1,y_2,\dots,y_k$ such that the edges $wv,wy_1,\dots,wy_k$ are in clockwise order.
	
Then $y_iy_j$ is an edge only if $j=i+1$.  Indeed, for any $1\leq \ell\leq k$ there exists a path from $v$ to $y_\ell$ (through $u$ and some $x\in X$) which avoids $\{w\}\cup Y\setminus \{y_\ell\}$.  But if $y_iy_j$ is an edge for some $j>i+1$, then the triangle $wy_iy_j$ separates $y_{i+1}$ from $v$.
	
Now let $y=y_i$ for some $i$.

For large enough $n$, $|N(y)\cap N(u)|\geq n^{5/6}\geq 7n^{4/5}$. Thus, by Corollary \ref{emptyk27}, there are vertices $x_1,\dots,x_7\in N(y)\cap N(u)$ (in natural order) which together with $u$ and $y$ form an empty (not necessarily induced) $K_{2,7}$.

By assumption, $x_4$ is contained in at least $2n/3-10$ induced $C_5$'s.  However, note that any such induced $C_5$ also contains $u$ and $y$.  Let $Z$ be the set of all vertices in $X\cup Y\setminus \{y,x_4\}$ which are contained in an induced $C_5$ containing $x_4$. By Lemma \ref{basicbound}, $|Z|+|V(G)\setminus (X\cup Y)|\geq 2n/3-10$. Since $|X\cup Y|\geq 2n/3-10$, it follows that $|Z|\geq n/4$. If $z\in Z\cap Y$, then $z$ is not a neighbour of $u$, so it must be a neighbour of $y_i$, hence $z=y_{i-1}$ or $z=y_{i+1}$. It follows that $|Z\cap X|\geq n/5$.
	
\medskip
	
\begin{claim}
	Either $y_{i-1}$ is a neighbour of $y_i$ and $|N(y_{i-1})\cap X|\geq n^{5/6}$ or $y_{i+1}$ is a neighbour of $y_i$ and $|N(y_{i+1})\cap X|\geq n^{5/6}$. \label{claim1}
\end{claim}
	
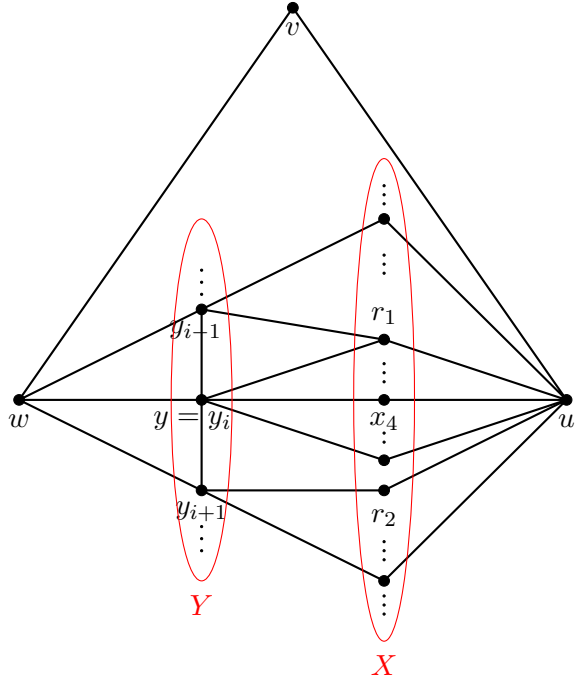
\begin{figure}
\centering
\begin{tikzpicture}[scale=0.4]
\draw[fill=black](0,0)circle(5pt);
\draw[fill=black](6,3)circle(5pt);
\draw[fill=black](6,0)circle(5pt);
\draw[fill=black](6,-3)circle(5pt);
\draw[fill=black](18,0)circle(5pt);
\draw[fill=black](12,4.2)circle(1pt);
\draw[fill=black](12,4.8)circle(1pt);
\draw[fill=black](12,4.5)circle(1pt);
\draw[fill=black](12,6)circle(5pt);
\draw[fill=black](12,0)circle(5pt);
\draw[fill=black](12,2)circle(5pt);
\draw[fill=black](12,-5.3)circle(1pt);
\draw[fill=black](12,-5.0)circle(1pt);
\draw[fill=black](12,-4.7)circle(1pt);
\draw[fill=black](12,-6)circle(5pt);
\draw[fill=black](12,-3)circle(5pt);
\draw[fill=black](12,-2)circle(5pt);
\draw[fill=black](9,13)circle(5pt);
\draw[fill=black](12,0.6)circle(1pt);
\draw[fill=black](12,0.9)circle(1pt);
\draw[fill=black](12,1.2)circle(1pt);
\draw[fill=black](12,-1.1)circle(1pt);
\draw[fill=black](12,-1.4)circle(1pt);
\draw[fill=black](6,3.5)circle(1pt);
\draw[fill=black](6,3.9)circle(1pt);
\draw[fill=black](6,4.3)circle(1pt);
\draw[fill=black](6,-5)circle(1pt);
\draw[fill=black](6,-4.2)circle(1pt);
\draw[fill=black](6,-4.6)circle(1pt);
\draw[fill=black](12,7.1)circle(1pt);
\draw[fill=black](12,6.8)circle(1pt);
\draw[fill=black](12,6.5)circle(1pt);
\draw[fill=black](12,-7.1)circle(1pt);
\draw[fill=black](12,-6.5)circle(1pt);
\draw[fill=black](12,-6.8)circle(1pt);
\draw[thick](0,0)--(6,3);
\draw[thick](0,0)--(6,0);
\draw[thick](0,0)--(6,-3);
\draw[thick](0,0)--(9,13)--(18,0)(6,3)--(12,6)--(18,0);
\draw[thick](6,-3)--(12,-6)--(18,0)(6,-3)--(12,-3)--(18,0);
\draw[thick](6,0)--(12,-2)--(18,0)(6,0)--(12,0)--(18,0)(6,3)--(12,2)--(18,0);
\draw[thick](6,0)--(12,2);
\draw[thick](6,-3)--(6,0);
\draw[thick](6,3)--(6,0);
\node at (0,-0.7) {$w$};
\node at (9,12.3) {$v$};
\node at (18,-0.7) {$u$};
\node at (5.8,2.3) {$y_{i-1}$};
\node at (5.7,-0.7)  {$y=y_i$};
\node at (6,-3.7) {$y_{i+1}$};
\node at (12,2.8)  {$r_1$};
\node at (12,-3.9)  {$r_2$};
\node at (12,-0.7)  {$x_4$};

\draw[rotate around={90:(6,0)},red] (6,0) ellipse (6 and 1);
\draw[rotate around={90:(12,0)},red] (12,0) ellipse (8 and 1);
\node[red] at (6,-6.8)  {$Y$};
\node[red] at (12,-8.8)  {$X$};
\end{tikzpicture}
\label{fig:claim1}
\caption{Proof of Claim \ref{claim1}}
\end{figure}
	
\begin{proof}
	We prove this claim by using $|Z\cap X|\geq n/5$.
	
	If $k=1$, then $X\s N(y)\cap N(u)$, so $Z=\emptyset$, which is a contradiction.
	
	Suppose that $k\geq 2$.  Let $z\in Z\cap X$. Observe that since $y,x_4,u$ and $z$ are contained in an \emph{induced} $C_5$, we have $z\not \in N(y)$, and the fifth vertex in the $C_5$ is some $q\in N(y)\cap N(z)$.
	
	Let us first assume that $2\leq i\leq k-1$.  Let $r_1$ be the element in $N(y_{i-1})\cap X$ which is the first neighbour of $u$ in clockwise order compared to $x_4$. Similarly, let $r_2$ be the element in $N(y_{i+1})\cap X$ which is the first neighbour of $u$ in anticlockwise order compared to $x_4$.  Note that the edges $wy_{i-1}$, $y_{i-1}r_1$, $r_1u$, $ur_2$, $r_2y_{i+1}$, $y_{i+1}w$ divide the plane into two regions; let $R$ be the one which contains $y$.  Then either $z$ is also in $R$ (possibly on the boundary), or $q$ is on the boundary of $R$. In the former case, there are only two possibilities for $z$: $r_1$ and $r_2$ (since $z\not \in N(y_i)$). Assume that $q$ is on the boundary of $R$. Since $ux_4yqz$ is an induced $C_5$, we have $q\not \in N(u)$.  Thus, $q\not \in X$ so $q\neq r_1$ and $q\neq r_2$.  Also, $z\in X$, so $z\not \in N(w)$, hence $q\neq w$.  Moreover, $q$ is distinct from $u$.  Thus, $q\in \{y_{i-1},y_{i+1}\}$, $q$ is a neighbour of $y_i$ and $z$ is a neighbour of $q$, from which the claim follows.
	
	Assume now that $i=1$. Let $r$ be the element in $N(y_2)\cap X$ which is the first neighbour of $u$ in anticlockwise order compared to $x_4$. The edges $wv,vu,ur,ry_2$ and $y_2w$ divide the plane into two regions; let $R$ be the one containing $y_1$.  Then either $z$ is also in $R$ (possibly on the boundary), or $q$ is on the boundary of $R$. In the former case, $z$ must be $r$ because $z\not \in N(y_1)$. Assume that $q$ is on the boundary of $R$. Note that $q\not \in N(u)$, so $q\neq r$ and $q\neq v$.  Also, $z\in X$, so $z\not \in N(w)$, hence $q\neq w$.  Moreover, $q$ is distinct from $u$.  Thus, $q=y_2$ and $z\in N(y_2)$.
	
	The case $i=k$ is very similar, so the claim is proved.  
	\end{proof}
	
	Assume that $y_{i-1}$ is a neighbour of $y_i$ and $|N(y_{i-1})\cap X|\geq n^{5/6}$. In particular, $|N(y_{i-1})\cap N(u)|\geq n^{5/6}$, so by Corollary \ref{emptyk27} there exists an empty $K_{2,7}$ whose part of size two consists of $y_{i-1}$ and $u$. This means that we can choose $w_1=w$, $w_2=y_{i-1}$ and $w_3=y_i$ and these vertices have the desired properties. The case where $y_{i+1}$ is a neighbour of $y_i$ with $|N(y_{i+1})\cap X|\geq n^{5/6}$ is almost identical.
\end{proof}

\section{The fine structure of extremal graphs} \label{sec:three vx removal}

In this section, we prove that if there do not exist three vertices whose removal deletes fewer than $h(n)-h(n-3)$ induced $5$-cycles, then the graph has a very specific structure (very close to the extremal graph $H_n$).

\begin{lemma} \label{lem:compare n and n-3}
    For any $n\geq 10$, $h(n)-h(n-3)=2n-11$.
\end{lemma}

\begin{proof}
Choose $k_1,k_2,k_3\in \mathbb{N}$ such that $k_1+k_2+k_3=n-4$ and $k_1,k_2,k_3$ are as close as possible. Then $h(n)=k_1k_2+k_2k_3+k_3k_1+2$ and $h(n-3)=(k_1-1)(k_2-1)+(k_2-1)(k_3-1)+(k_3-1)(k_1-1)+2$, so
$h(n)-h(n-3)=2k_1+2k_2+2k_3-3=2(n-4)-3=2n-11$.
\end{proof}

\begin{lemma} \label{lem:three vx removal}
    Let $G$ be a planar graph on $n$ vertices. Let $w_1,w_2,w_3$ be vertices in $G$ forming a triangle. Let $u$ be a vertex, not in $N(w_1)\cup N(w_2)\cup N(w_3)$, such that for every $i\in \{1,2,3\}$, there exists some $v_i\in N(u)\cap N(w_i)$ which is not a neighbour of $v_j$ and $w_j$ for $j\neq i$. Then the number of induced $C_5$'s containing $uv_iw_i$ for some $1\leq i\leq 3$ is at most $2n-11$. Moreover, if equality holds, then every vertex $x\in V(G)\setminus \{u,w_1,w_2,w_3\}$ satisfies one of the following.
    \begin{enumerate}[label=(\roman*)]
        \item There exists some $i$ such that $x\in N(w_i)\cap N(u)$, or
        \item there exist some $i\neq j$ such that $x\in N(w_i)\cap N(w_j)\setminus N(u)$, $N(x)\cap N(w_i)\cap N(u)\neq \emptyset$ and $N(x)\cap N(w_j)\cap N(u)\neq \emptyset$, or \label{prop:upgraded}
        \item there exist some $i\neq j$ and $y\in N(w_i)\cap N(w_j)\setminus N(u)$ such that $x\in N(u)\cap N(y)\setminus (N(w_1)\cup N(w_2)\cup N(w_3))$.
    \end{enumerate}
\end{lemma}

The following lemma will be used in the proof of Lemma \ref{lem:three vx removal}.

\begin{lemma} \label{lem:sum X and Y}
    Let $G$ be a planar graph on $n$ vertices. Let $w_1,w_2,w_3$ be vertices in $G$ forming a triangle. Let $u$ be a vertex, not in $N(w_1)\cup N(w_2)\cup N(w_3)$, such that for every $i\in \{1,2,3\}$, there exists some $v_i\in N(u)\cap N(w_i)$ which is not a neighbour of $w_j$ for $j\neq i$. For every $i\in \{1,2,3\}$, let $X_i$ and $Y_i$ be defined like $X$ and $Y$ in Lemma \ref{basicbound} with $w_i$ and $v_i$ taking the role of $w$ and $v$. Let $G_i$ be the induced bipartite subgraph of $G$ with parts $X_i$ and $Y_i$. For $i\in \{1,2,3\}$, let $\lambda_i=1$ if $w_{i+1}$ and $w_{i+2}$ are in the same connected component in $G_i$ (where temporarily we write $w_4:=w_1$ and $w_5:=w_2$) and let $\lambda_i=0$ otherwise. Then
    $$\sum_{i\leq 3} (|X_i|+|Y_i|)\leq 2(n-1)-\lambda_1-\lambda_2-\lambda_3.$$
    Moreover, if equality holds, then every vertex $x\in V(G)\setminus \{u,w_1,w_2,w_3\}$ satisfies one of the following.
    \begin{enumerate}[label=(\roman*)]
        \item There exists some $i$ such that $x\in N(w_i)\cap N(u)$, or
        \item there exist some $i\neq j$ such that $x\in N(w_i)\cap N(w_j)\setminus N(u)$ and $N(x)\cap N(u)\neq \emptyset$, or \label{prop:in Y}
        \item there exist some $i\neq j$ and $y\in N(w_i)\cap N(w_j)\setminus N(u)$ such that $x\in N(u)\cap N(y)\setminus (N(w_1)\cup N(w_2)\cup N(w_3))$. \label{prop:in X}
    \end{enumerate}
\end{lemma}

\begin{proof}
    First, note that for every $i\in \{1,2,3\}$, $u$ does not belong to either of $X_i$ and $Y_i$. We now want to show that every $x\in V(G)$ belongs to at most two of the sets $X_1,X_2,X_3,Y_1,Y_2,Y_3$. For this, first observe that no $x\in V(G)$ belongs to both $X_i$ and $Y_j$ for some $i,j\in \{1,2,3\}$. This is because elements in $X_i$ are neighbours of $u$, while elements in $Y_j$ are non-neighbours. Hence, it remains to prove that no $x\in V(G)$ belongs to $X_1\cap X_2\cap X_3$ or to $Y_1\cap Y_2\cap Y_3$. If $x\in Y_1\cap Y_2\cap Y_3$, then $x\in N(w_1)\cap N(w_2)\cap N(w_3)$, but then $x$ cannot have a common neighbour with $u$ by planarity, which is a contradiction. If $x\in X_1\cap X_2\cap X_3$, then $x\in N(u)\setminus (N(w_1)\cup N(w_2)\cup N(w_3))$ and $x$ has a common neighbour with each $w_i$, but this is again impossible by planarity.
    
    It already follows that $\sum_{i\leq 3} (|X_i|+|Y_i|)\leq 2(n-1)$.
    
    Assume now that for some $i\in \{1,2,3\}$, the vertices $w_{i+1}$ and $w_{i+2}$ belong to the same connected component in $G_i$. It is not hard to see that this implies that $N(w_{i+1})\cap N(w_{i+2})\cap N(u)\neq \emptyset$. Then this common neighbour does not belong to any $Y_j$ (since it is a neighbour of $u$), but it also does not belong to $X_{i+1}$ and $X_{i+2}$ (where, again, indices are understood modulo 3). So for every such index $i$, we ``gain 1" compared to $\sum_{i\leq 3} (|X_i|+|Y_i|)\leq 2(n-1)$. Since $N(w_1)\cap N(w_2)\cap N(w_3)\cap N(u)=\emptyset$, it follows that
    \begin{equation} \label{eqn:sum X and Y}
        \sum_{i\leq 3} (|X_i|+|Y_i|)\leq 2(n-1)-\lambda_1-\lambda_2-\lambda_3.
    \end{equation}
    
    Assume now that equality holds here. Then every $x\in V(G)\setminus \{u\}$ which is not in $N(u)\cap (N(w_1)\cup N(w_2)\cup N(w_3))$ must belong to two of the sets $X_1,X_2,X_3,Y_1,Y_2,Y_3$. Assume first that $x\in Y_i\cap Y_j$ for some $i\neq j$. Then $x\in N(w_i)\cap N(w_j)$ and $N(x)\cap N(u)\neq \emptyset$, so $x$ satisfies \ref{prop:in Y}. Suppose now that $x\in X_i\cap X_j$ for some $i\neq j$. Then $x\in N(u)\setminus (N(w_i)\cup N(w_j))$ and $x$ has neighbours $y_i\in Y_i$ and $y_j\in Y_j$. It follows that $x$ does not have a common neighbour with $w_k$, where $w_k$ is the member of $\{w_1,w_2,w_3\}$ different from $w_i$ and $w_j$ since otherwise $G$ contains a subdivision of $K_5$ (the vertices of the $K_5$ are $u,x,w_1,w_2,w_3$), contradicting planarity. But since we have equality in (\ref{eqn:sum X and Y}), $y_i$ must belong to two of $Y_i$, $Y_j$ and $Y_k$. Hence, we necessarily have $y_i\in Y_i\cap Y_j$. It follows that $y_i\in N(w_i)\cap N(w_j)$. So we may take $y=y_i$ and then $x$ satisfies property \ref{prop:in X}.
\end{proof}

\begin{proof}[Proof of Lemma \ref{lem:three vx removal}]
    For every $i\in \{1,2,3\}$, define $X_i$, $Y_i$, $G_i$ and $\lambda_i$ as in Lemma \ref{lem:sum X and Y} and let $\mu_i$ be the number of connected components of $G_i$. By Lemma \ref{basicbound}, the number of induced $C_5$'s containing $u$, $v_i$ and $w_i$ is at most $|X_i|+|Y_i|-\mu_i$. Since for every $i\neq j$, there exists an induced $C_5$ with vertices $u,v_i,w_i,v_j,w_j$, it follows that the number of induced $C_5$'s containing $uv_iw_i$ for some $1\leq i\leq 3$ is at most
    $$\left(\sum_{i\leq 3} (|X_i|+|Y_i|-\mu_i)\right)-3.$$
    By Lemma \ref{lem:sum X and Y}, this is at most $2n-5-(\mu_1+\mu_2+\mu_3)-(\lambda_1+\lambda_2+\lambda_3)$. Note that for every $i\in \{1,2,3\}$, $\mu_i+\lambda_i\geq 2$. Hence, $\mu_1+\mu_2+\mu_3+\lambda_1+\lambda_2+\lambda_3\geq 6$ and it follows that the number of induced $C_5$'s containing $uv_iw_i$ for some $i$ is at most $2n-11$.
    
    Assume now that this number is precisely $2n-11$. Then we must have
    $$\sum_{i\leq 3}(|X_i|+|Y_i|)=2(n-1)-\lambda_1-\lambda_2-\lambda_3$$ and
    $$\mu_1+\mu_2+\mu_3+\lambda_1+\lambda_2+\lambda_3=6.$$
    The first equation implies, using Lemma \ref{lem:sum X and Y}, that every $x\in V(G)\setminus \{u,w_1,w_2,w_3\}$ satisfies one of the following.
    \begin{enumerate}[label=(\alph*)]
        \item There exists some $i$ such that $x\in N(w_i)\cap N(u)$, or
        \item there exist some $i\neq j$ such that $x\in N(w_i)\cap N(w_j)\setminus N(u)$ and $N(x)\cap N(u)\neq \emptyset$, or \label{prop:to upgrade}
        \item there exist some $i\neq j$ and $y\in N(w_i)\cap N(w_j)\setminus N(u)$ such that $x\in N(u)\cap N(y)\setminus (N(w_1)\cup N(w_2)\cup N(w_3))$. 
    \end{enumerate}
    This is almost what we need; it just remains to prove that if some $x\in V(G)\setminus \{w_1,w_2,w_3\}$ satisfies property \ref{prop:to upgrade}, then it also satisfies property \ref{prop:upgraded} in the statement of this lemma. Assume that it does not; then WLOG $N(x)\cap N(w_i)\cap N(u)=\emptyset$. Note that from the proof of Lemma~\ref{lem:sum X and Y} it is clear that we must have $x\in Y_j$ to attain equality in (\ref{eqn:sum X and Y}). It is not hard to see that $N(x)\cap N(w_i)\cap N(u)=\emptyset$ implies that $x$ and $w_i$ are in different connected components of $G_j$. Since $x$ and $w_k$ (where $w_k$ is the element of $\{w_1,w_2,w_3\}$ different from $w_i$ and $w_j$) are also in different connected components in $G_j$, it follows that $\mu_j+\lambda_j\geq 3$ and hence
    $$\mu_1+\mu_2+\mu_3+\lambda_1+\lambda_2+\lambda_3\geq 7,$$ which is a contradiction.
\end{proof}

\begin{lemma} \label{lem:two scenarios}
    Let $G$ be a plane graph on $n$ vertices. Let $w_1,w_2,w_3$ be vertices in $G$ forming a triangle. Let $u$ be a vertex, not in $N(w_1)\cup N(w_2)\cup N(w_3)$, such that for every $i\in \{1,2,3\}$, there exists an empty $K_{2,7}$ whose part of size two is $\{u,w_i\}$. Assume that there do not exist three vertices such that the number of induced $C_5$'s containing at least one of these three vertices is at most $2n-12$. Then one of the following two scenarios must occur.
    \begin{enumerate}[label=(\alph*)]
        \item Every $x\in V(G)\setminus \{u,w_1,w_2,w_3\}$ is a common neighbour of $u$ and $w_i$ for some $i\in \{1,2,3\}$, or \label{prop:no X}
        \item after relabelling $w_1$, $w_2$ and $w_3$ if necessary, there exists $w_4\in N(w_2)\cap N(w_3)\setminus (N(u)\cup \{w_1\})$ such that $N(w_2)\cap N(w_4)\cap N(u)\neq \emptyset$, $N(w_4)\cap N(w_3)\cap N(u)\neq \emptyset$, and every $x\in V(G)\setminus \{u,w_1,w_2,w_3,w_4\}$ belongs to $N(w_i)\cap N(u)$ for some $i\in \{1,2,3,4\}$. \label{prop:there is X}
    \end{enumerate}
\end{lemma}

\begin{proof}
    For each $i\in \{1,2,3\}$, take an empty $K_{2,7}$ and call its centre vertex $v_i$. It is easy to see that any induced $C_5$ in $G$ which contains $v_i$ must also contain $u$ and $w_i$. Hence, by Lemma \ref{lem:three vx removal}, the number of induced $C_5$'s in $G$ containing at least one of $v_1$, $v_2$, $v_3$ is at most $2n-11$. By the assumption in our lemma, it must therefore be precisely $2n-11$. Hence, by Lemma \ref{lem:three vx removal} again, every vertex $x\in V(G)\setminus \{u,w_1,w_2,w_3\}$ satisfies one of the following.
    \begin{enumerate}[label=(\roman*)]
        \item There exists some $i\in \{1,2,3\}$ such that $x\in N(w_i)\cap N(u)$, or \label{prop:standard}
        \item there exist distinct $i,j\in \{1,2,3\}$ such that $x\in N(w_i)\cap N(w_j)\setminus N(u)$, $N(x)\cap N(w_i)\cap N(u)\neq \emptyset$ and $N(x)\cap N(w_j)\cap N(u)\neq \emptyset$, or \label{prop:in Y new}
        \item there exist distinct $i,j\in \{1,2,3\}$ and $y\in N(w_i)\cap N(w_j)\setminus N(u)$ such that $x\in N(u)\cap N(y)\setminus (N(w_1)\cup N(w_2)\cup N(w_3))$. \label{prop:in X new}
    \end{enumerate}
    Suppose first that there is no vertex $x\in V(G)\setminus \{u,w_1,w_2,w_3\}$ which satisfies \ref{prop:in X new}. If, in addition, there is no $x\in V(G)\setminus \{u,w_1,w_2,w_3\}$ satisfying \ref{prop:in Y new} either, then scenario \ref{prop:no X} holds and we are done. Moreover, if there is only one vertex $x\in V(G)\setminus \{u,w_1,w_2,w_3\}$ satisfying \ref{prop:in Y new}, then we can choose that vertex to be $w_4$ and scenario \ref{prop:there is X} is satisfied (after relabelling $w_1$, $w_2$ and $w_3$ if necessary). Assume that there are at least two vertices $x\in V(G)\setminus \{u,w_1,w_2,w_3\}$ satisfying \ref{prop:in Y new}. Call them $y$ and $y'$, and assume, without loss of generality, that $y\in N(w_1)\cap N(w_2)$ and $y'\in N(w_2)\cap N(w_3)$.
    
    If $N(w_1)\cap N(w_2)\cap N(u)\neq \emptyset$, then the only neighbours of $y$ are vertices of a triangle, and hence $y$ is not contained in any induced $5$-cycle. In this case, the number of induced $5$-cycles containing one of $v_1,v_2,y$ is at most (in fact, substantially less than) $2n-12$, a contradiction. Similarly, we cannot have $N(w_2)\cap N(w_3)\cap N(u)\neq \emptyset$.
    
    \begin{claim}
        The number of induced $C_5$'s containing at least one of $y$, $v_1$ and $v_2$ is at most $2n-12$.
    \end{claim}
    
    \begin{proof}
    Let $k_1=|N(w_1)\cap N(u)|$ and let $k_2=|N(w_2)\cap N(u)|$. It is not hard to see that the induced $C_5$'s containing $y$ are $ux_1w_1yz$ for some $x_1\in N(w_1)\cap N(u)\setminus N(y)$ and for the unique $z\in N(y)\cap N(w_2)\cap N(u)$ and $ux_2w_2yz$ for some $x_2\in N(w_2)\cap N(u)\setminus N(y)$ and for the unique $z\in N(y)\cap N(w_1)\cap N(u)$. There are $k_1+k_2-2$ such induced $C_5$'s. Every induced $C_5$ containing $v_2$ also contains $u$ and $w_2$, and by Lemma \ref{basicbound}, the number of induced $C_5$'s containing $uv_2w_2$ is at most $n-(k_2+2)-1$ (since $X\cup Y\subset V(G)\setminus \left((N(u)\cap N(w_2))\cup \{u,w_2\}\right)$ when we take $v=v_2,w=w_2$ in the lemma). Furthermore, by the same lemma, the number of induced $C_5$'s containing $uv_1w_1$ is at most $n-(k_1+2)-2$. This is because we have seen (before the claim) that $N(w_2)\cap N(w_3)\cap N(u)=\emptyset$, and hence $\mu\geq 2$ when we apply Lemma \ref{basicbound} with $v=v_1,w=w_1$.
    
    Combining our estimates and noting that for any two of $y$, $v_1$ and $v_2$, there exists an induced $C_5$ containing those two vertices but not the third, it follows that the number of induced $C_5$'s containing at least one of $y$, $v_1$ and $v_2$ is at most $(k_1+k_2-2)+(n-k_2-3)+(n-k_1-4)-3=2n-12$, as claimed.
    \end{proof}
    
    The claim contradicts the conditions set out in the lemma, so we may assume that there does exist a vertex $x\in V(G)\setminus \{u,w_1,w_2,w_3\}$ which satisfies \ref{prop:in X new}.
    Hence, after reordering $w_1$, $w_2$ and $w_3$ if necessary, we may assume that there exist $v_4,w_4\in V(G)$ such that $w_4\in N(w_2)\cap N(w_3)\setminus N(u)$ and $v_4\in N(u)\cap N(w_4)\setminus (N(w_1)\cup N(w_2)\cup N(w_3))$. Since $v_4\in N(w_4)\setminus N(w_1)$, we have $w_4\neq w_1$. Also, $w_4\not \in N(u)$, so it cannot satisfy properties \ref{prop:standard} and \ref{prop:in X new}, hence it must satisfy \ref{prop:in Y new}. It is not hard to see that therefore $N(w_4)\cap N(w_2)\cap N(u)\neq \emptyset$ and $N(w_4)\cap N(w_3)\cap N(u)\neq \emptyset$. Write $r$ for the unique vertex in $N(w_4)\cap N(w_2)\cap N(u)$ and write $r'$ for the unique vertex in $N(w_4)\cap N(w_3)\cap N(u)$. Let $R$ be the region bounded by edges $ur,rw_4,w_4r',r'u$ and containing $v_4$. By the classification of the vertices in $V(G)\setminus \{u,w_1,w_2,w_3\}$, any vertex in the interior of $R$ must be a common neighbour of $w_4$ and $u$. It follows that the only possible induced $C_5$ containing $v_4$ which does not contain $u$ and $w_4$ is $v_4rw_2w_3r'$ (these vertices may or may not induce a $C_5$).
    
    We will now use Lemma \ref{lem:three vx removal} to upper bound the number of induced $C_5$'s containing $uv_iw_i$ for some $2\leq i\leq 4$. The lemma can be applied with $w_4$ and $v_4$ in place of $w_1$ and $v_1$ and it follows that there are at most $2n-11$ induced $C_5$'s containing $uv_iw_i$ for some $2\leq i\leq 4$. Assume, for contradiction, that scenario \ref{prop:there is X} in the statement of Lemma \ref{lem:two scenarios} does not hold. In this case, we can improve the $2n-11$ bound.
    
    \begin{claim} \label{claim:sum of three}
        The number of induced $C_5$'s containing $uv_iw_i$ for some $2\leq i\leq 4$ is at most $2n-13$.
    \end{claim}
    
    \begin{proof}
    Since scenario \ref{prop:there is X} does not hold, there exists some $x\in V(G)\setminus \{u,w_1,w_2,w_3,w_4\}$ which does not belong to $N(u)\cap N(w_i)$ for any $i\in \{1,2,3,4\}$. By the classification of the vertices in $V(G)\setminus \{u,w_1,w_2,w_3\}$, this implies that either there exists a vertex $y\in N(w_1)\cap N(w_2)\setminus (N(u)\cup \{w_3\})$ such that $N(y)\cap N(w_1)\cap N(u)\neq \emptyset$ and $N(y)\cap N(w_2)\cap N(u)\neq \emptyset$, or there exists a vertex $y\in N(w_1)\cap N(w_3)\setminus (N(u)\cup \{w_2\})$ such that $N(y)\cap N(w_1)\cap N(u)\neq \emptyset$ and $N(y)\cap N(w_3)\cap N(u)\neq \emptyset$. WLOG, assume that the former holds. As before, $N(w_1)\cap N(w_2)\cap N(u)=\emptyset$. Let $S=N(y)\cap N(u)\setminus (N(w_1)\cup N(w_2))$. Note that there does not exist an induced $C_5$ containing all of $u$, $v_4$, $w_4$ and at least one element in $S\cup \{y\}$ and that there does not exist an induced $C_5$ containing all of $u$, $v_3$, $w_3$ and at least one element in $S\cup \{y\}$. Moreover, it is not hard to see that the number of induced $C_5$'s containing $u$, $v_2$, $w_2$ and at least one element of $S\cup \{y\}$ is at most $|S|+1$. By Lemma \ref{lem:three vx removal} applied to the planar graph $G-(S\cup \{y\})$, the number of induced $C_5$'s containing $uv_iw_i$ for some $2\leq i\leq 4$ but none of $S\cup \{y\}$ is at most $2(n-|S|-1)-11$. Combining this with our previous estimate, we conclude that the number of induced $C_5$'s containing $uv_iw_i$ for some $2\leq i\leq 4$ is at most $2n-|S|-12$. This is at most $2n-12$, so we are done unless it is exactly $2n-12$. In this case, the number of induced $C_5$'s in $G-(S\cup \{y\})$ containing $uv_iw_i$ for some $2\leq i\leq 4$ is exactly $2(n-|S|-1)-11$. By the equality case in Lemma \ref{lem:three vx removal} (applied to the graph $G-(S\cup \{y\})$ and with $v_4,w_4$ replacing $v_1,w_1$) and since $w_1\not \in N(u)$, it follows that $N(w_1)\cap N(w_2)\cap N(u)\neq \emptyset$, which is a contradiction.
    \end{proof}
    The claim implies that the number of induced $C_5$'s containing $v_2$, $v_3$ or $v_4$ is at most $2n-12$, a contradiction.
\end{proof}

\section{Completing the proof of Proposition \ref{prop:combine all}} \label{sec:finishing}

\begin{lemma} \label{lem:scenarios okay}
    In the setting of Lemma \ref{lem:two scenarios}, $G$ contains at most $h(n)$ induced $C_5$'s.
\end{lemma}

\begin{proof}
We treat scenarios \ref{prop:no X} and \ref{prop:there is X} separately.

Let us assume first that scenario \ref{prop:no X} holds. Let $s_1=N(w_1)\cap N(u)\setminus (N(w_2)\cup N(w_3))$ and define $s_2,s_3$ analogously. For $(i,j)\in \{(1,2),(2,3),(3,1)\}$, let $\gamma_{i,j}=|N(w_i)\cap N(w_j)\cap N(u)|$. Note that for every $i,j$, $\gamma_{i,j}\in \{0,1\}$ and that $s_1+s_2+s_3+\gamma_{1,2}+\gamma_{2,3}+\gamma_{3,1}=n-4$. The induced $C_5$'s in $G$ are those of the form $ux_iw_iw_jx_j$ where $i,j$ are distinct elements of $\{1,2,3\}$, $x_i\in N(w_i)\cap N(u)\setminus N(w_j)$ and $x_j\in N(w_j)\cap N(u)\setminus N(w_i)$. For $(i,j)=(1,2)$, the number of such induced $C_5$'s is $(s_1+\gamma_{3,1})(s_2+\gamma_{2,3})$. Combining this with the analogous formulae for the cases $(i,j)=(2,3)$ and $(i,j)=(3,1)$, we conclude that the number of induced $C_5$'s in $G$ is
\begin{equation}
    (s_1+\gamma_{3,1})(s_2+\gamma_{2,3})+(s_2+\gamma_{1,2})(s_3+\gamma_{3,1})+(s_3+\gamma_{2,3})(s_1+\gamma_{1,2}). \label{eqn:scenario a count}
\end{equation}
Setting $\delta_1=\gamma_{3,1}+\gamma_{1,2}-\gamma_{2,3}$, $\delta_2=\gamma_{1,2}+\gamma_{2,3}-\gamma_{3,1}$ and $\delta_3=\gamma_{2,3}+\gamma_{3,1}-\gamma_{1,2}$, we get that (\ref{eqn:scenario a count}) is equal to
$$(s_1+\delta_1)(s_2+\delta_2)+(s_2+\delta_2)(s_3+\delta_3)+(s_3+\delta_3)(s_1+\delta_1)+\gamma_{3,1}\gamma_{2,3}+\gamma_{1,2}\gamma_{3,1}+\gamma_{2,3}\gamma_{1,2}-\delta_1\delta_2-\delta_2\delta_3-\delta_3\delta_1.$$
It is straightforward to check that $\gamma_{3,1}\gamma_{2,3}+\gamma_{1,2}\gamma_{3,1}+\gamma_{2,3}\gamma_{1,2}-\delta_1\delta_2-\delta_2\delta_3-\delta_3\delta_1\leq 1$. Hence, the number of induced $C_5$'s in $G$ is at most $(s_1+\delta_1)(s_2+\delta_2)+(s_2+\delta_2)(s_3+\delta_3)+(s_3+\delta_3)(s_1+\delta_1)+1$. Since $s_1+\delta_1+s_2+\delta_2+s_3+\delta_3=s_1+s_2+s_3+\gamma_{1,2}+\gamma_{2,3}+\gamma_{3,1}=n-4$, it follows that the number of induced $C_5$'s in $G$ is at most
$$\max\{k_1k_2+k_2k_3+k_3k_1: k_1,k_2,k_3\in \mathbb{N}, k_1+k_2+k_3=n-4\}+1<h(n).$$

Let us assume now that scenario \ref{prop:there is X} holds. Let $s_1=|N(w_1)\cap N(u)\setminus (N(w_2)\cup N(w_3)\cup N(w_4))|$, and define $s_2,s_3,s_4$ analogously. Moreover, let $\gamma_{3,1}=|N(w_3)\cap N(w_1)\cap N(u)|$ and let $\gamma_{1,2}=|N(w_1)\cap N(w_2)\cap N(u)|$. Note that $s_1+s_2+s_3+s_4+\gamma_{3,1}+\gamma_{1,2}=n-7$. One can check that every induced $C_5$ in $G$ contains $u$ unless $s_4=1$ in which case there is a unique induced $C_5$ not containing $u$, which consists of $w_2$, $w_3$, the unique vertex in $N(w_3)\cap N(w_4)\cap N(u)$, the unique vertex in $N(w_4)\cap N(u)\setminus (N(w_2)\cup N(w_3))$ and the unique vertex in $N(w_2)\cap N(w_4)\cap N(u)$. Every induced $C_5$ containing $u$ contains precisely two of the $w_i$'s. The number of those containing $w_1$ and $w_2$ is $(s_1+\gamma_{3,1})(s_2+1)$; the number of those containing $w_2$ and $w_3$ is $(s_2+1+\gamma_{1,2})(s_3+1+\gamma_{3,1})$; the number of those containing $w_3$ and $w_1$ is $(s_3+1)(s_1+\gamma_{1,2})$; the number of those containing $w_2$ and $w_4$ is $(s_2+\gamma_{1,2})(s_4+1)$; and the number of those containing $w_4$ and $w_3$ is $(s_4+1)(s_3+\gamma_{3,1})$ (there is none which contains $w_1$ and $w_4$ since those vertices are nonadjacent). Altogether, the number of induced $C_5$'s in $G$ is at most
\begin{align}
&(s_1+\gamma_{3,1})(s_2+1)+(s_2+1+\gamma_{1,2})(s_3+1+\gamma_{3,1})+(s_3+1)(s_1+\gamma_{1,2}) \nonumber \\
&+(s_2+\gamma_{1,2})(s_4+1)+(s_4+1)(s_3+\gamma_{3,1})+1. \label{eqn:scenario b count}
\end{align}
The coefficient of $s_1$ is $(s_2+1)+(s_3+1)$, while the coefficient of $s_4$ is $(s_2+\gamma_{1,2})+(s_3+\gamma_{3,1})$. Hence, (\ref{eqn:scenario b count}) is at most
$$(s_1+s_4+\gamma_{3,1})(s_2+1)+(s_2+1+\gamma_{1,2})(s_3+1+\gamma_{3,1})+(s_3+1)(s_1+s_4+\gamma_{1,2})+(s_2+\gamma_{1,2})+(s_3+\gamma_{3,1})+1.$$
Setting $\delta_1=\gamma_{1,2}+\gamma_{3,1}+1$, $\delta_2=\gamma_{1,2}-\gamma_{3,1}+1$ and $\delta_3=\gamma_{3,1}-\gamma_{1,2}+1$, the above sum is equal to
\begin{align*}
    &(s_1+s_4+\delta_1)(s_2+\delta_2)+(s_2+\delta_2)(s_3+\delta_3)+(s_3+\delta_3)(s_1+s_4+\delta_1)+2 \\
    &+3\gamma_{3,1}+3\gamma_{1,2}+\gamma_{3,1}\gamma_{1,2}-\delta_1\delta_2-\delta_2\delta_3-\delta_3\delta_1.
\end{align*}
It is straightforward to check that $3\gamma_{3,1}+3\gamma_{1,2}+\gamma_{3,1}\gamma_{1,2}-\delta_1\delta_2-\delta_2\delta_3-\delta_3\delta_1\leq 0$, hence the number of induced $C_5$'s in $G$ is at most
$$(s_1+s_4+\delta_1)(s_2+\delta_2)+(s_2+\delta_2)(s_3+\delta_3)+(s_3+\delta_3)(s_1+s_4+\delta_1)+2.$$
But $(s_1+s_4+\delta_1)+(s_2+\delta_2)+(s_3+\delta_3)=s_1+s_2+s_3+s_4+\gamma_{1,2}+\gamma_{3,1}+3=n-7+3=n-4$, so the number of induced $C_5$'s in $G$ is at most
$$\max\{k_1k_2+k_2k_3+k_3k_1: k_1,k_2,k_3\in \mathbb{N}, k_1+k_2+k_3=n-4\}+2= h(n),$$
as claimed.
\end{proof}

We can now combine several of our lemmas to prove Proposition \ref{prop:combine all}.

\begin{proof}[Proof of Proposition \ref{prop:combine all}]
Let $n$ be sufficiently large and let $G$ be an $n$-vertex plane graph with $h(n)+t$ induced $5$-cycles for some $t\geq 1$. Assume, for contradiction, that every subgraph on $n-1$ vertices has at most $h(n-1)+t$ induced $5$-cycles and that every subgraph on $n-3$ vertices has at most $h(n-3)+t$ induced $5$-cycles. Then every vertex in $G$ is contained in at least $h(n)-h(n-1)$ induced $5$-cycles. Moreover, for any three vertices, the number of induced $5$-cycles containing at least one of the three vertices is at least $h(n)-h(n-3)$. By Lemma \ref{lem:compare n and n-1}, we have $h(n)-h(n-1)>2n/3-4$ and by Lemma \ref{lem:compare n and n-3}, we have $h(n)-h(n-3)=2n-11$. Since $n$ is sufficiently large, $h(n)+t>\frac{5}{18}n^2$ and $h(n)-h(n-1)>2n/3-10$, Lemma \ref{lem:three empty k27s} implies that there exist distinct vertices $w_1$, $w_2$, $w_3$ and $u$ such that $w_1w_2,w_2w_3,w_3w_1\in E(G)$ and for every $1\leq i\leq 3$, there exists an empty $K_{2,7}$ whose part of size two is $\{w_i,u\}$. Note that for each $i\in \{1,2,3\}$, we have $u\not \in N(w_i)$ (else the centre vertex in the empty $K_{2,7}$ with $u$ and $w_i$ forming the part of size two is contained in no induced $C_5$'s). Using that for any three vertices, the number of induced $5$-cycles containing at least one of these vertices is at least $2n-11$, Lemma~\ref{lem:scenarios okay} implies that $G$ contains at most $h(n)$ induced $5$-cycles, which is a contradiction.
\end{proof}

\section{Concluding remarks}

In this paper we determined precisely the maximum possible number of induced $5$-cycles that an $n$-vertex planar graph can contain. As remarked in the introduction, the maximum number of induced triangles is known exactly, while the maximum number of induced $4$-cycles is known asymptotically. Very recently, Cox and Martin \cite{CM21old,CM21new} showed that for $k\in \{3,4,5,6\}$, the maximum possible number of (not necessarily induced) $2k$-cycles in an $n$-vertex planar graph is $(\frac{n}{k})^k+o(n^k)$. A construction attaining this bound is obtained by taking a $2k$-cycle and blowing up every second vertex to a set of size roughly $n/k$. The resulting planar graph contains $(\frac{n}{k})^k+o(n^k)$ $2k$-cycles and all of them are induced, which shows that for $k\in \{3,4,5,6\}$, the maximum number of induced $2k$-cycles is also $(\frac{n}{k})^k+o(n^k)$. Cox and Martin conjecture that their result can be extended to general $k$.

\begin{conjecture}[Cox and Martin \cite{CM21old}]
    For any $k\geq 3$, the maximum possible number of (not necessarily induced) $2k$-cycles in an $n$-vertex planar graph is $(\frac{n}{k})^k+o(n^k)$.
\end{conjecture}

If true, this would imply that the maximum number of induced $2k$-cycles is also $(\frac{n}{k})^k+o(n^k)$. Turning to odd cycles, the situation seems to be a bit more complicated. Indeed, if we take a $(2k+1)$-cycle and blow up $k$ pairwise non-adjacent vertices to sets of size roughly $n/k$, then the resulting planar graph contains $(\frac{n}{k})^k+o(n^k)$ induced $C_{2k+1}$'s. On the other hand, there are constructions with much more (not necessarily induced) $C_{2k+1}$'s. Indeed, we can blow up every second vertex of a $2k$-cycle to sets of size $n/k$ and take a spanning path inside each blownup set. The resulting planar graph will contain $2k(\frac{n}{k})^k+o(n^k)$ copies of $C_{2k+1}$, but they will not be induced. Hence, it remains possible that for $k\geq 3$, the maximum number of induced $(2k+1)$-cycles is $(\frac{n}{k})^k+o(n^k)$ (but by our results this is not true when $k=2$).

\paragraph{Note added.} Building on the original version of this paper where we proved an asymptotic form of Theorem \ref{maxindc5}, but independently from our revision, Michael Savery \cite {Sav21} also determined the exact maximum possible number of induced $5$-cycles in an $n$-vertex planar graph when $n$ is sufficiently large, and characterised the extremal examples. He also proved that for sufficiently large $n$, $K_{2,n-2}$ is the unique $n$-vertex planar graph maximising the number of induced $4$-cycles.

\section*{Acknowledgment}
We thank the two anonymous referees for their valuable comments which significantly improved the paper. We are also grateful to Michael Savery for pointing out a small error in the previous version.

The research of the second, fifth and sixth author was partially supported by the National Research, Development and Innovation Office NKFIH, grants  K116769, K132696. The research of the third author was partially supported by the NKFIH grant 2018-2.1.7-UK\_GYAK-2019-00020. The research of the fifth author is partially supported by  Shota Rustaveli National Science Foundation of Georgia SRNSFG, grant number DI-18-118.

\bibliographystyle{abbrv}
\bibliography{reference}
\end{document}